\documentclass[a4paper,reqno]{amsart}
\usepackage[T1]{fontenc}
\usepackage[utf8x]{inputenc}
\usepackage[english]{babel}
\usepackage{amsmath,amsfonts,amssymb}
\usepackage{graphicx}
\usepackage{hyperref}
\usepackage{enumerate}
\usepackage{xcolor}

\usepackage{amsthm}

\theoremstyle{definition}
\newtheorem{definition}{Definition}[section]

\theoremstyle{plain}
\newtheorem*{D*}{Definition}
\newtheorem*{theorem*}{Theorem}
\newtheorem*{prop*}{Proposition}

\newtheorem{thmA}{Theorem}

\newtheorem{corA}[thmA]{Corollary}

\newtheorem{theorem}{Theorem}[section]
\newtheorem{lemma}[theorem]{Lemma}
\newtheorem{cor}[theorem]{Corollary}
\newtheorem{proposition}[theorem]{Proposition}

\theoremstyle{definition}
\newtheorem{remark}[theorem]{Remark}

\newcommand{\R}{\ensuremath{\mathbb R}}
\newcommand{\N}{\ensuremath{\mathbb N}}

\newcommand{\s}{\ensuremath{\mathbb S}}
\newcommand{\eps}{\ensuremath{\varepsilon}}

\newcommand{\meas}{\mathfrak{m}}

\DeclareMathOperator{\CD}{CD}
\DeclareMathOperator{\RCD}{RCD}

\usepackage{xcolor}

\title{Sphere theorems for $\RCD$ and stratified spaces}
\author{Shouhei Honda}
\address{Mathematical Institute
Tohoku University
Sendai 980-8578
Japan}
\email{shonda@m.tohoku.ac.jp}

\author{Ilaria Mondello}
\address{Universit\'e de Paris-Est Cr\'eteil, Laboratoire d’Analyse et math\'ematiques appliqu\'ees}
\email{ilaria.mondello@u-pec.fr}

\begin{document}

\nocite{*}

\maketitle
\begin{abstract}
We prove topological sphere theorems for $\RCD(n-1, n)$ spaces which generalize Colding's results and Petersen's result to the $\RCD$ setting. We also get an improved sphere theorem in the case of Einstein stratified spaces. 
\end{abstract}
\section*{Introduction}
In \cite{Colding} Colding proved that if a closed $n$-dimensional Riemannian manifold $(M^n, g)$ with $\mathrm{Ric}_{M^n}^g\ge n-1$ satisfies that the radius is close to $\pi$, then $M^n$ is homeomorphic to the standard $n$-dimensional unit sphere $\mathbb{S}^n$, where the \textit{radius} $\mathrm{rad}(X, d)$ of a metric space $(X, d)$ is defined by:
\begin{equation}
\mathrm{rad} (X, d) = \inf_{x \in X} \sup_{y \in X} d(x, y).
\end{equation}
Thanks to Cheeger-Colding's work \cite{CheegerColding1}, it is known that this homeomorphism can be improved to the diffeomorphism. 
Our main results generalize the previous to a large class of singular spaces, the so-called \textit{$\RCD$ metric measure spaces}, or \textit{$\RCD$ spaces} for short, whose study is now quickly developping (see for instance  \cite{Ambrosio} by Ambrosio for a survey).


Roughly speaking, a metric measure space $(X, d, m)$ is said to be $\RCD(K, N)$ if, in a generalized sense, the Ricci curvature is bounded below by $K$, the dimension is bounded above by $N$ and the space carries some Riemannian structure (we refer to the first section for a precise definition, see Definition \ref{def:rcd}). One of the typical examples can be found in \textit{weighted Riemannian manifolds} $(M^n, d_g, e^{-f}\mu_g)$, where $d_g$ denotes the distance defined by the Riemannian metric $g$, $\mu_g$ is the Riemannian volume measure and $f$ is a smooth function on $M^n$. In fact $(M^n, d_g, e^{-f}\mu_g)$ is a $\RCD(K, N)$ space if and only if $n\le N$ and 
$$
\mathrm{Ric}_{M^n}^g +\mathrm{Hess}_f^g-\frac{df \otimes df}{N-n}\ge Kg
$$
hold. As easily noticed from this example, in the $\RCD$ theory, there is a \textit{flexibility} on the choice of reference measures even if the base metric space $(X, d)$ is fixed. In particular, the measure $m$ is not necessarily the Hausdorff measure associated to the distance. 

Other examples of $\RCD(K,N)$ spaces are given by compact \emph{stratified spaces} $(X^n, d_g,\mu_g)$ endowed with the distance and measure associated to an iterated edge metric $g$, under the suitable assumptions on $g$. Stratified  spaces are singular manifolds with \emph{iterated conical singularities}, isolated or not. When the metric $g$ has Ricci tensor bounded from below on the regular set and angles along the codimension 2 singular set are smaller than $2\pi$, $(X^n,d_g,\mu_g)$ is a $\RCD$ space, as proven in \cite{BKMR} by Bertrand-Ketterer-Richard and the second author. In this work, as a consequence of our main results, we will obtain a sphere theorem for \textit{Einstein} stratified spaces. It is worth pointing out that examples of Einstein stratified spaces occur in various branches of geometry: for instance in mathematical physics, the singular space associated to a static triple \cite{Ambrozio}; in Kähler geometry, Kähler-Einstein manifolds with edge singularities of cone angle smaller than $2\pi$ along a smooth divisor \cite{JMR}. 

We are now in a position to state the main result of the paper:

\begin{thmA}[Topological sphere theorem for $\RCD$ spaces, I]\label{mthm}
For all $n \in \mathbb{N}_{\ge 2}$ there exists a positive constant $\epsilon_n>0$ such that if a compact metric space $(X, d)$ satisfies that $\mathrm{rad}(X, d) \ge \pi- \epsilon_n$ and that $(X, d, \meas)$ is a $\RCD (n-1, n)$ space for some Borel measure $\meas$ on $X$ with full support, then $X$ is homeomorphic to the $n$-dimensional sphere. 
\end{thmA}

This seems the first topological sphere theorem in the $\RCD$ theory. We emphasize again that the theorem states that although there is a flexibility on the choice of $\meas$, the topological structure is uniquely determined. Note that in the previous theorem one cannot replace the radius by the diameter of the space. Indeed, for any $\varepsilon > 0$ Anderson constructed in \cite{Anderson} manifolds of even dimension $n \geq 4$, with Ricci tensor bounded below by $n-1$ and diameter larger than $\pi - \epsilon$, which are \emph{not} homeomorphic to the sphere. Similar examples can be found in \cite{Otsu} by Otsu.

In order to introduce an application, let us recall a result of Petersen \cite{Petersen}; for a closed $n$-dimensional Riemannian manifold $(M^n ,g)$ with $\mathrm{Ric}_{M^n}^g\ge n-1$ the following two conditions are equivalent quantitatively:
\begin{enumerate}
\item The $(n+1)$-th eigenvalue of the Laplacian is close to $n$
\item The radius is close to $\pi$.
\end{enumerate} 
In particular if the one of them above holds, then $M^n$ is diffeomorphic to $\mathbb{S}^n$.

Note that even in the $\RCD$-setting, the above equivalence is justified by the spectral convergence result of Gigli-Mondino-Savar\'e \cite{GigliMondinoSavare} and the rigidity results of Ketterer \cite{Ketterer2}. In particular we have the following;
\begin{corA}[Topological sphere theorem for $\RCD$ spaces, II]\label{thm:eigenhomeo}
For all $n \in \mathbb{N}_{\ge 2}$ there exists a positive constant $\epsilon_n>0$ such that if a $\RCD(n-1, n)$ space $(X, d, \meas)$ satisfies 
\begin{equation}\label{eq:maxeig}
\lambda_{n+1} \le n+\epsilon_n,
\end{equation}
then $X$ is homeomorphic to $\mathbb{S}^n$, where $\lambda_k:=\lambda_k(X, d, \meas)$ denotes the $k$-th eigenvalue of the (minus) Laplacian $-\Delta$ on $(X, d, \meas)$. 
\end{corA}
In the second section, we give a proof for reader's convenience.
In Corollary \ref{thm:eigenhomeo} it is known that if $(X, d, \meas)$ is a Riemannian manifold, that is, $(X, d, \meas)$ is isometric to $(M^n, d_g, \mu_g)$ for a closed Riemannian manifold $(M^n, g)$, then the assumption (\ref{eq:maxeig}) can be replaced by a weaker one;
\begin{equation}
\lambda_n \le n+\epsilon_n.
\end{equation}
See \cite{Aubry} by Aubry (see also \cite{Bert} by Bertrand and \cite{Honda09} by the first author). However in the $\RCD$ setting we can not get such improvement. In fact, the $n$-dimensional unit hemisphere $\mathbb{S}^n_+$ with the standard Riemannian measure is a $\RCD(n-1, n)$ space with $\lambda_k=n$ for all $1\le k \le n$, but it is not homeomorphic to $\mathbb{S}^n$. Thus Corollary \ref{thm:eigenhomeo} is sharp in this sense.

Let us explain how to prove the main theorem.
For that, we recall the original proof by Colding. First, he proved that if the radius is close to $\pi$, then the volume is almost maximal. Perelman's topological sphere theorem \cite{Perelman} for almost maximal volume then allows Colding to conclude. 

We follow a similar argument. However, the almost maximality of the volume does not make sense in the general setting of RCD spaces, because, as we pointed out above, there is flexibility in the choice of measures. In order to overcome this difficulty, the assumption in Theorem \ref{mthm} allows us to get rid of such flexibility of the measure $\meas$, in the following sense: we start by proving
\begin{equation}
\meas=\frac{\meas(X)}{\mathcal{H}^n(X)} \mathcal{H}^n,
\end{equation}
where $\mathcal{H}^n$ is the $n$-dimensional Hausdorff measure. This is justified by using a recent result of the first author \cite{Honda19} which confirms a conjecture by De Philippis-Gigli (see Remark 1.9 in \cite{DePhilippisGigli}) in the compact setting. 

Then by combining this with a compactness result for non-collapsed $\RCD$ spaces by De Philippis-Gigli \cite{DePhilippisGigli} and Ketterer's rigidity \cite{Ketterer}, we can show that our situation is reduced to the study of the following measured Gromov-Hausdorff convergent sequence of $\RCD(n-1, n)$ spaces:
$$
(X_i, d_i, \mathcal{H}^n) \stackrel{mGH}{\to} (\mathbb{S}^n, d_{\mathbb{S}^n}, \mathcal{H}^n).
$$
Then we can follow an argument similar to Colding's proof by using the intrinsic Reifenberg theorem \cite{CheegerColding1} by Cheeger-Colding instead of using Perelman's topological sphere theorem \cite{Perelman}.

One step in the previous proof consists in showing that the almost maximality of the Hausdorff measure implies that the space is homeomorphic to the sphere (see Theorem \ref{prop:toprigidity} in the following). The same assumption in the case of an \emph{Einstein} stratified space actually allows us to get a stronger sphere theorem, independently of the $\RCD$ theory:

\begin{corA}[Sphere theorem for Einstein stratified spaces]
For all $n \in \mathbb{N}_{\ge 2}$ there exists a positive constant $\epsilon_n>0$ such that the following holds. Let $(X, g)$ be  compact $n$-dimensional stratified space endowed with an iterated edge metric $g$ such that  $\mathrm{Ric}_g \equiv n-1$ on the regular set. Assume that there is no singular stratum of codimension 2. If $\mu_g(X) \ge (1-\epsilon_n)\mathcal{H}^n(\mathbb{S}^n)$, then $(X, d_g)$ is isometric to $(\mathbb{S}^n, d_{\mathbb{S}^n})$.
\end{corA}

The assumption on the codimension 2 stratum cannot be dropped. Indeed, consider a compact Riemannian surface $(X,g)$ with sectional curvature equal to one away from a finite number of isolated conical singularities of angles smaller than $2\pi$. Thanks to \cite{BKMR}, such surface is a $\RCD(1,2)$ space, and therefore the almost maximality of its volume implies that $(X,g)$ is homeomorphic to $\s^2$ and that the angles at the singularities are close to $2\pi$. Nevertheless, this cannot be improved to an isometry. The strategy of our proof actually consists in showing that there cannot be any singularity of codimension strictly greater than 2.  Moreover, a local almost maximality for the volume allows one to control the regularity in a neighbourhood of a point (see Corollary 3.5), even if the space is not compact. Both of these proofs do not depend on the space being a  $\RCD$ space. 

The paper is organized as follows: in the next section we give a quick introduction about $\RCD$ spaces and recall related results we need later. 
In Section $2$, after preparing few technical results, we prove the main results. The last section is devoted to showing Corollary C in the case of Einstein stratified spaces, for which we state the basic notions that we need.

After we finalized this article we learned that the paper \cite{KM} by Kapovitch-Mondino contains the same results as in Corollary \ref{thm:eigenhomeo} and Theorem \ref{prop:toprigidity} as an independent work. Their proofs are close to ours, but the scopes of our works differ: while they are  mainly concerned with the topology and the boundary of non-collapsed $\RCD$ spaces, our work also intends to discuss sphere theorems in the more specific setting of stratified spaces.

\smallskip\noindent
\textbf{Acknowledgement.}
The authors would like to thank Takumi Yokota for informing the smooth version of Theorem \ref{thm:strati} and Alexander Lytchak for his comments concerning Alexandrov spaces. The second author would like to thank Erwann Aubry for useful discussions. The first author acknowledges supports of the Grantin-Aid
for Young Scientists (B) 16K17585 and Grant-in-Aid for Scientific Research (B) of 18H01118.

\section{Preliminary}
We say that a triple $(X, d, \meas)$ is a \textit{metric measure space} if $(X, d)$ is a complete separable metric space and $\meas$ is a Borel measure on $X$ with full support, that is, $\meas (B_r(x))>0$ holds for all $x \in X$ and all $r>0$, where $B_r(x)$ denotes the open ball centered at $x$ of radius $r$.
 
Throughout this section we fix $K \in \mathbb{R}$, $N \in (1, \infty)$ and $n \in \mathbb{N}_{\ge 2}$.
\subsection{General $\RCD$ space}
Let us fix a metric measure space $(X, d, \meas)$. 
The goal of this section is to give a quick introduction on $\RCD$ spaces with their fundamental properties.
The Cheeger energy
$\mathrm{Ch}=:L^2(X,\meas)\to [0,+\infty]$ is a convex and $L^2(X,\meas)$-lower semicontinuous functional defined as follows:
\begin{equation}\label{eq:defchp}
\mathrm{Ch}(f):=\inf\left\{\liminf_{n\to\infty}\frac{1}{2}\int_X(\mathrm{Lip}  f_n)^2d\meas:\ \text{$f_n\in\mathrm{Lip}_b (X,d)\cap L^2(X, \meas)$, $\|f_n-f\|_{L^2}\to 0$}\right\}, 
\end{equation}
where $\mathrm{Lip} f$ denotes the local Lipschitz constant and $\mathrm{Lip}_b(X,d)$ is the
space of bounded Lipschitz functions. 
The Sobolev space $H^{1,2}(X,d,\meas)$ then coincides with the domain of the Cheeger energy, that is $\{f\in L^2(X,m):\ \mathrm{Ch}(f)<+\infty\}$. When endowed with the norm
$$
\|f\|_{H^{1,2}}:=\left(\|f\|_{L^2(X,\meas)}^2+2\mathrm{Ch}(f)\right)^{1/2}
$$
this space is Banach and  
separable Hilbert if $\mathrm{Ch}$ is a quadratic form (see \cite{AmbrosioGigliSavare14}). 
According to the terminology introduced in \cite{Gigli1}, we say that a 
metric measure space $(X,d,\meas)$ is \textit{infinitesimally Hilbertian} if $\mathrm{Ch}$ is a quadratic form.
  
By looking at minimal relaxed slopes and by a polarization procedure, one can then define a {\it carr\'e du champ}
$$
\Gamma:H^{1,2}(X,d,\meas)\times H^{1,2}(X,d,\meas)\rightarrow L^1(X,\meas)
$$
playing in this abstract theory the role of the scalar product between gradients. In infinitesimally Hilbertian metric measure
spaces, the $\Gamma$ operator
satisfies all natural symmetry, bilinearity, locality and chain rule properties, and provides integral representation to $\mathrm{Ch}$:
$$ 2\mathrm{Ch}(f)=\int_X \Gamma(f,f)\,d\meas,
$$ 
for all $f\in H^{1,2}(X,d,\meas)$. 

We can now define a densely
defined operator $\Delta:D(\Delta)\to L^2(X,\meas)$ whose domain consists of all functions $f\in H^{1,2}(X,d,\meas)$
satisfying
$$
 \int_X hgd\meas=-\int_X \Gamma(f,h)d\meas\quad\qquad\forall h\in H^{1,2}(X,d,\meas)
$$
for some $g\in L^2(X,\meas)$. The unique $g$ with this property is then denoted by $\Delta f$ (see \cite{AmbrosioGigliSavare13}).

We are now in a position to introduce the definition of $\RCD$ spaces.
\begin{definition}[$\RCD$ spaces]\label{def:rcd}
For $K \in \mathbb{R}$ and $\hat{N} \in [1, \infty]$, $(X, d, \meas)$ is said to be a \textit{$\RCD(K, \hat{N})$ space} if the following are satisfied: 
\begin{enumerate}
\item Infinitesimally Hilbertian: it is inifinitesimally Hilbertian;
\item Volume growth: there exist $x \in X$ and $c>1$ such that $\meas (B_r(x)) \le Ce^{Cr^2}$ for all $r>0$;
\item Sobolev-to-Lipschitz property : any  
$f\in H^{1,2}(X,d,\meas)$ with $\Gamma (f, f) \leq 1$ $\meas$-a.e. in $X$ 
has a $1$-Lipschitz representative.
\item Bakry-\'Emery inequality : for all $f\in D(\Delta)$ with $\Delta f\in H^{1,2}(X,d,\meas)$, 
\begin{equation}\label{eq:boch}
\frac{1}{2}\int_X \Gamma (f, f)\Delta \phi d\meas\ge
\int_X\phi\left(\frac{(\Delta f)^2}{\hat N}+ \Gamma (f, \Delta f)+ K\Gamma (f, f)\right)d\meas 
\end{equation}
for all $\phi\in D(\Delta) \cap L^{\infty}(X, \meas)$ with $\phi\geq 0$ and $\Delta\phi\in L^\infty(X,\meas)$.
\end{enumerate}
\end{definition}
It is worth pointing out that a $\RCD(K,N)$ space was originally defined as a metric measure space which is infinitesimally Hilbertian and satisfies the $\CD(K,N)$ condition in the sense of Lott-Sturm-Villani (see 
\cite{AmbrosioGigliSavare14, AmbrosioMondinoSavare, ErbarKuwadaSturm, Gigli1, LottVillani, Sturm1, Sturm2}). Such definition has been proven to be equivalent to the formulation given by Definition  \ref{def:rcd} (see also \cite{CavMil} for the equivalence between $\RCD$ and $\RCD^*$).

In order to keep our presentation short, we skip the definitions of \textit{pointed measured Gromov-Hausdorff convergence} (pmGH), of \textit{measured Gromov-Hausdorff convergence} (mGH), and of \textit{pointed Gromov-Hausdorff convergence} (pGH). We refer to \cite{CheegerColding1, Fukaya, GigliMondinoSavare, Sturm1, Sturm2} for the precise definitions. Note that the radius is continuous with respect to the Gromov-Hausdorff convergence.

Let us introduce a compactness result for $\RCD$ spaces with respect to the pmGH convergence, which follows from \cite[Cor.3.22, Thm.7.2]{GigliMondinoSavare}.
\begin{theorem}[Compactness of $\RCD$ spaces]\label{thm:comprcd}
Let $(X_i, d_i, \meas_i, x_i) (i=1, 2, \ldots)$ be a sequence of pointed $\RCD(K, N)$ spaces. If there exist $v_i >0(i=1, 2)$ with $v_1\le v_2$ such that $v_1\le \meas_i(B_1(x_i))\le v_2$ holds for all $i$, then there exist a subsequence $(X_{i(j)}, d_{i(j)}, \meas_{i(j)}, x_{i(j)})$ and a pointed $\RCD (K, N)$ space $(X, d, \meas, x)$ such that 
$$(X_{i(j)}, d_{i(j)}, \meas_{i(j)}, x_{i(j)}) \stackrel{pmGH}{\to} (X, d, \meas, x),$$
that is, $(X_{i(j)}, d_{i(j)}, \meas_{i(j)}, x_{i(j)})$ pmGH converge to $(X, d, \meas, x)$. 
\end{theorem}

The next definition is a key notion of the paper. Inspired by the stratification of Ricci limit spaces given by Cheeger-Colding theory, one can define a $k$-regular set as the set of points for which the tangent cone, in the sense of pmGH convergence, is the Euclidean space $\R^k$. More precisely we have:
\begin{definition}[Regular set]
Let $(X, d, \meas)$ be a $\RCD (K, N)$ space and let $k \in \mathbb{N}$.
Then the \textit{$k$-dimensional regular set $\mathcal{R}_k=\mathcal{R}_k(X)$} is defined by the set of all points $x \in X$ satisfying that
\begin{equation}
(X, r^{-1}d, (\meas (B_r(x)))^{-1}\meas, x) \stackrel{pmGH}{\to} (\mathbb{R}^k, d_{\mathbb{R}^k}, \omega_k^{-1}\mathcal{L}^k, 0_k) \quad (r \to 0^+).
\end{equation}
\end{definition}
By the Bishop-Gromov inequality (see \cite{LottVillani, Sturm1, Sturm2, Villani}) it is easy to check that $\mathcal{R}_k=\emptyset$ for all $k \in (N, \infty) \cap \mathbb{N}$.
\begin{remark}
It is proved that for $x \in X$ and $k \in \mathbb{N}$, if $(X, r^{-1}d, x) \stackrel{pGH}{\to} (\mathbb{R}^k, d_{\mathbb{R}^k}, 0_k)$, then $x \in \mathcal{R}_k$. The proof is same as the one of \cite[Prop.1.35]{CheegerColding1}. Thus the $k$-dimensional regular set is a purely metric notion in this sense.
\end{remark}

The following theorem is proved in \cite[Thm.0.1]{BrueSemola} (after \cite[Cor.1.2]{MondinoNaber}). It generalizes a result of \cite[Thm.1.18]{ColdingNaber} to $\RCD$ spaces and allows one to define a unique \emph{essential dimension} for a $\RCD$ space.
\begin{theorem}[Essential dimension]
Let $(X, d, \meas)$ be a $\RCD (K, N)$ space. Assume that $X$ is not a single point.
Then there exists a unique $k:=\mathrm{dim}_{d, \meas}(X) \in \mathbb{N} \cap [1, N]$ such that $\meas (X \setminus \mathcal{R}_k)=0$. We call it the \textit{essential dimension} of $(X, d, \meas)$. 
\end{theorem}
We end this subsection by introducing a fundamental property on the essential dimension proved in \cite[Thm.1.5]{Kita};
\begin{theorem}[Lower semicontinuity of essential dimensions]\label{thm:lower}
The essential dimension is lower semicontinuous with respect to the pointed measured Gromov-Hausdorff convergence of $\RCD(K, N)$ spaces.
\end{theorem}
\subsection{Rigidity for positively Ricci curved $\RCD$ spaces}
It is worth pointing out that in general if a $\RCD(K ,N)$ space $(X, d, \meas)$ has a bounded diameter, then $(X, d)$ must be compact and the spectrum of the (minus) Laplacian $-\Delta$ is discrete and unbounded;
\begin{equation}
0=\lambda_0<\lambda_1 \le \lambda_2 \le \cdots \to \infty,
\end{equation} 
where $\lambda_i$ denotes the $i$-th eigenvalue counted with multiplicities (see for instance \cite{GigliMondinoSavare} or \cite{AmbrosioHondaTewodrosePortegies}).
Let us introduce some properties of $\RCD(N-1, N)$ spaces with the rigidity results we will use later (\cite[Cor.1.3, Cor.1.6]{Ketterer}, \cite[Thm.1.4]{Ketterer2}).
\begin{theorem}[Rigidity to the sphere]\label{thm:ket}
Let $(X, d, \meas)$ be a $\RCD(N-1, N)$ space. Then the following are satisfied:
\begin{enumerate}
\item The diameter $\mathrm{diam}(X, d)$ is at most $\pi$ (in particular $\mathrm{rad}(X, d) \le \pi$) and the first positive eigenvalue $\lambda_1$ is at least $N$.
\item If $\mathrm{rad}(X, d)=\pi$ or $\lambda_{[N]+1}=N$, where $[N]$ is the integer part of $N$, then $N$ must be an integer, $(X, d)$ is isometric to $(\mathbb{S}^N, d_{\mathbb{S}^N})$ and $\meas=a \mathcal{H}^N$ for some $a>0$.  
\end{enumerate}
\end{theorem}
The next two corollaries give us reasons for only discussing the case of integer $n$ in Theorem \ref{mthm} and Corollary \ref{thm:eigenhomeo}.
\begin{cor}
For all $N \in [1, \infty) \setminus \mathbb{N}$, there exists a positive constant $\tau_N>0$ such that any $\RCD(N-1, N)$ space $(X, d, \meas)$ satisfies
$\lambda_{[N]+1} \geq N + \tau_N$.
\end{cor}
\begin{proof}
The proof is done by a contradiction. If the statement is not satisfied, then there exists a sequence of $\RCD(N-1, N)$ spaces $(X_i, d_i, \meas_i)$ such that 
\begin{equation}\label{eq:gap}
\lim_{i \to \infty}\lambda_{[N]+1}(X_i, d_i, \meas_i)=N.
\end{equation}
By Theorem \ref{thm:comprcd} with no loss of generality we can assume that $(X_i, d_i, \meas_i)$ mGH-converge to a $\RCD(N-1, N)$ space $(X, d, \meas)$. Then the spectral convergence result proved in \cite[Thm.7.8]{GigliMondinoSavare} with (\ref{eq:gap}) yields $$\lambda_{[N]+1}(X, d, \meas)=\lim_{i \to \infty}\lambda_{[N]+1}(X_i, d_i, \meas_i)=N.$$ Theorem \ref{thm:ket} shows that $N$ must be an integer, which is a contradiction.
\end{proof}
Similarly we have the following.
\begin{cor}
For all $N \in [1, \infty) \setminus \mathbb{N}$, there exists a positive constant $\delta_N>0$ such that any $\RCD(N-1, N)$ space $(X, d, \meas)$ satisfies
$\mathrm{rad}(X, d) \leq \pi - \delta_N$.
\end{cor}
\subsection{Non-collapsed $\RCD$ space}
Let us introduce a special class of $\RCD$ spaces introduced in \cite{DePhilippisGigli}, that is non-collapsed and weakly non-collapsed $\RCD$ spaces. The non-collapsing assumption means that the measure $\meas$ is chosen to coincide, or to be absolutely continuous, with respect to the Hausdorff measure. More precisely we have: 

\begin{definition}[Non-collapsed $\RCD$ space]
Let $(X, d, \meas)$ be a $\RCD(K, N)$ space.
\begin{enumerate}
\item $(X, d, \meas)$ is called \textit{non-collapsed} if $\meas=\mathcal{H}^N$.
\item $(X, d, \meas)$ is called \textit{weakly non-collapsed} if $\meas \ll \mathcal{H}^N$.
\end{enumerate}
\end{definition}
The following fundamental results for non-collapsed $\RCD(K, N)$ spaces are proved in \cite{DePhilippisGigli};
\begin{theorem}[Fine properties of non-collapsed $\RCD$ spaces]\label{thm:noncol}
Let $(X, d, \mathcal{H}^N)$ be a non-collapsed $\RCD(K, N)$ space. Then the following are satisfied;
\begin{enumerate}
\item $N$ must be an integer;
\item For all $x \in X$ the Bishop inequality holds in the sense of 
\begin{equation}\label{eq:bishop0}
\lim_{r \to 0^+}\frac{\mathcal{H}^N(B_r(x))}{\mathrm{Vol}_{K, N}(r)}\le 1,
\end{equation}
where $\mathrm{Vol}_{K, N}(r)$ denotes the volume of a ball of radius $r$ in the $N$-dimensional space form whose Ricci curvature is constant equal to $K$. Moreover the equality in (\ref{eq:bishop0}) holds if and only if $x \in \mathcal{R}_N$.
\end{enumerate}
\end{theorem}
Let us give several remarks on the theorem above. The first property (1) is also true for weakly non-collapsed $\RCD(K, N)$ spaces.
The second property (2) can be regarded as a rigidity result. Moreover it is proven that the \textit{almost} rigidity of (\ref{eq:bishop0}) also holds, and this will play a role in the next section. See \cite[Thm.1.6]{DePhilippisGigli} for the precise statement (see also \cite[Prop.6.6]{AmbrosioHondaTewodrosePortegies}).

The next theorem follows from a combination of \cite{Kita} and \cite{DePhilippisGigli}.
For the reader's convenience we give a proof:
\begin{theorem}\label{thm:topdim}
For $n \in \mathbb{N}_{\ge 2}$, a $\RCD(K, n)$ space $(X, d, \meas)$ is weakly non-collapsed if and only if $\mathcal{R}_n \neq \emptyset$.
\end{theorem}
\begin{proof}
We check only the ``if'' part because  the ``only if'' part is a direct consequence of \cite[Thm.1.10]{DePhilippisGigli}.
Let $x \in \mathbb{R}_n$. Then since Theorem \ref{thm:lower} yields
$$
\liminf_{r \to 0^+}\mathrm{dim}_{ r^{-1}d, (\meas (B_r(x)))^{-1}\meas}(X) \ge \mathrm{dim}_{d_{\mathbb{R}^n}, \omega_n^{-1}\mathcal{L}^n}(\mathbb{R}^n)=n,
$$
we have $\mathrm{dim}_{d, \meas}(X)=n$ because $\mathrm{dim}_{d, \meas}(X)=\mathrm{dim}_{ r^{-1}d, (\meas (B_r(x)))^{-1}\meas}(X)$ for all $r>0$.
Then \cite[Thm.1.10]{DePhilippisGigli} yields that $(X, d, \meas)$ is weakly non-collapsed.
\end{proof}
Let us introduce one of the main results of \cite{Honda19} which confirmed a conjecture raised in \cite{DePhilippisGigli} in the compact case (\cite[Cor.1.4]{Honda19});
\begin{theorem}[``Weakly non-collapsed'' implies ``non-collapsed'']\label{thm:weak}
Let $(X, d, \meas)$ be a compact weakly non-collapsed $\RCD(K, n)$ space. Then 
$$
\meas =\frac{\meas (X)}{\mathcal{H}^n(X)}\mathcal{H}^n.
$$
\end{theorem}
To conclude this section, we introduce a compactness result for non-collapsed $\RCD$ spaces with respect to pmGH convergence, which is proved in \cite[Thm.1.2, Thm.1.3]{DePhilippisGigli} (Compare with Theorem \ref{thm:comprcd});
\begin{theorem}[Compactness of non-collapsed $\RCD$ spaces]\label{thm:compnon}
Let $(X_i, d_i, \mathcal{H}^n, x_i) (i=1, 2, \ldots)$ be a sequence of pointed non-collapsed $\RCD(K, n)$ spaces with $$\liminf_{i \to \infty}\mathcal{H}^n(B_1(x_i))>0.$$ Then there exist a subsequence $(X_{i(j)}, d_{i(j)}, \mathcal{H}^n, x_{i(j)})$ and a pointed non-collapsed $\RCD (K, n)$ space $(X, d, \mathcal{H}^n, x)$ such that 
$$(X_{i(j)}, d_{i(j)}, \mathcal{H}^n, x_{i(j)}) \stackrel{pmGH}{\to} (X, d, \mathcal{H}^n, x).$$
\end{theorem}
\section{Proof of main results}
Throughout the section we fix $n \in \mathbb{N}_{\ge 2}$ and $K \in \mathbb{R}$ too. In the next proposition let us consider Sturm's $\mathbb{D}$-distance between compact $\RCD(K, n)$ spaces, that is, ``$\epsilon$-mGH-close'' means ``$\mathbb{D}<\epsilon$'' below. Note that the convergence with respect to $\mathbb{D}$ is equivalent to the mGH convergence for compact $\RCD(K, N)$ spaces (see \cite{Sturm1, Sturm2}).
\begin{proposition}[Noncollapsed $\RCD$ is an open condition]\label{prop:open}
Let $(X, d_X, \meas_X)$ be a compact weakly non-collapsed $\RCD(K, n)$ space. Then there exists a positive constant $\epsilon_0=\epsilon_0(X, d_X, \meas_X)>0$ such that if a compact $\RCD(K, n)$ space $(Y, d_Y, \meas_Y)$ is $\epsilon_0$-mGH-close to $(X, d_X, \meas_X)$, then 
\begin{equation}\label{eq:measrigid}
\meas_Y=\frac{\meas_Y(Y)}{\mathcal{H}^n(Y)}\mathcal{H}^n.
\end{equation}
\end{proposition}
\begin{proof}
Let us prove the proposition by contradiction. If the statement does not hold,  then there exists a mGH convergent sequence $(X_i, d_i, \meas_i)$ of $\RCD(K, n)$ spaces to $(X, d_X, \meas_X)$ such that
\begin{equation}\label{eq:contr}
\meas_{X_i}\neq \frac{\meas_i(X_i)}{\mathcal{H}^n(X_i)}\mathcal{H}^n.
\end{equation}
However since Theorem \ref{thm:lower} states
$$
\liminf_{i \to \infty}\mathrm{dim}_{d_i, \meas_i}(X_i) \ge \mathrm{dim}_{d_X, \meas_X}(X)=n,
$$
we see that $\mathrm{dim}_{d_i, \meas_i}(X_i)=n$ for any sufficiently large $i$.
In particular Theorem \ref{thm:topdim} yields that $(X_i, d_i, \meas_i)$ is weakly non-collapsed $\RCD(K, n)$ space for such $i$. Then by applying Theorem \ref{thm:weak} to such $(X_i, d_i, \meas_i)$ we obtain that $\meas_i=\frac{\meas_i(X_i)}{\mathcal{H}^n(X_i)}\mathcal{H}^n$ which contradicts (\ref{eq:contr}).
\end{proof}
Let us remark that the Bishop inequality (\ref{eq:bishop0}) implies
\begin{equation}\label{eq:bishop}
\mathcal{H}^n(X)\le \mathcal{H}^n(\mathbb{S}^n)
\end{equation}
for all non-collapsed $\RCD(n-1, n)$ space $(X, d, \mathcal{H}^n)$. Since the equality easily implies $\mathrm{rad} (X, d)=\pi$, thus by Theorem (\ref{thm:ket}) the equality holds if and only if the space is isometric to the round sphere $\mathbb{S}^n$.
\begin{theorem}[Topological sphere theorem for $\RCD$ spaces, III]\label{prop:toprigidity}
There exists a positive constant $\epsilon_n>0$ such that if a compact non-collapsed $\RCD(n-1, n)$ space $(X, d, \mathcal{H}^n)$ satisfies $\mathcal{H}^n(X) \ge (1- \epsilon_n)\mathcal{H}^n(\mathbb{S}^n)$, then $X$ is homeomorphic to $\mathbb{S}^n$.
\end{theorem}

\begin{remark}
We point out that the previous result is known for Alexandrov spaces, see for example Proposition A.9 in the work of Yamaguchi \cite{Yamaguchi}. It can also be easily proved by contradiction, by combining the rigidity of Bishop-Gromov inequality and the topological stability theorem \cite{Per} for Alexandrov spaces (see \cite{KV07}). Moreover for $n=2$, the work \cite{LS18} showed that a non-collapsed $\RCD(K,2)$ space is an Alexandrov space with curvature at least $K$. As a consequence, our result directly holds in dimension $2$. We give here the proof in full generality. 
\end{remark}

\begin{proof}
The proof is done by contradiction. Assume that the theorem does not hold, then there exist sequences $\epsilon_i \rightarrow 0$ and $(X_i, d_i, \mathcal{H}^n)$ of $\RCD(n-1, n)$ spaces such that $\mathcal{H}^n(X_i)\geq (1-\epsilon_i)\mathcal{H}^n(\s^n)$ and $X_i$ is not homeomorphic to $\mathbb{S}^n$. Then we have 
\begin{equation}\label{eq:volconv}
\lim_{i \to \infty}\mathcal{H}^n(X_i)=\mathcal{H}^n(\mathbb{S}^n).
\end{equation}
and it is not difficult to check that (\ref{eq:volconv}) implies that $\mathrm{rad}(X_i, d_i) \to \pi$. Applying Theorem \ref{thm:ket} with Theorem \ref{thm:compnon} yields 
\begin{equation}
(X_i, d_i, \mathcal{H}^n) \stackrel{mGH}{\to} (\mathbb{S}^n, d_{\mathbb{S}^n}, \mathcal{H}^n).
\end{equation}
On the other hand the inequality
\begin{equation}\label{eq:max}
\frac{\mathcal{H}^n(X_i)}{\mathcal{H}^n(\mathbb{S}^n)} \ge 1- \epsilon_i,
\end{equation} 
together with the Bishop (\ref{eq:bishop0}) and the Bishop-Gromov inequalities implies that for all $x \in X_i$ we have
\begin{equation}\label{eq:maxball}
\frac{\mathcal{H}^n(B_r(x))}{\mathcal{H}^n(B_r(p))}\ge 1- \epsilon_i, \quad \forall r \in (0, \pi],\, \forall p \in \mathbb{S}^n.
\end{equation}
Applying this observation with the almost rigidity on the Bishop inequality \cite[Thm.1.5]{DePhilippisGigli} (see also \cite[Prop.6.6]{AmbrosioHondaTewodrosePortegies}) implies that for all $\eps >0$ there exists a positive integer $i_0 \in \N$ such that for all $i \geq i_0$ and $r \in (0,\pi]$
\begin{equation}
d_{GH}(B_{r/2}(x_i),B_{r/2}(0^n))\leq \eps r,
\end{equation}
where $d_{GH}$ denotes the Gromov-Hausdorff distance. This means that for all $i \geq i_0$, the metric spaces $\{(X_i, d_i)\}_i$ are uniformly Reifenberg flat. Then applying the intrinsic Reifenberg theorem \cite[Thm.A.1.2 and Thm.A.1.3]{CheegerColding1} yields that $X_i$ is homeomorphic to $\mathbb{S}^n$ for any sufficiently large $i$, which is a contradiction.
\end{proof}

We are in position to prove our main result.

\begin{proof}[Proof of Theorem \ref{mthm}]
The proof is done by a contradiction. Assume that the statement is false: then there exists a sequence $(X_i, d_i, \meas_i)$ of $\RCD(n-1, n)$ spaces such that $\mathrm{rad}(X_i, d_i) \to \pi$, $\meas_i(X_i)=1$ and $X_i$ is not homeomorphic to $\mathbb{S}^n$. By Theorem \ref{thm:ket} with no loss generality we can assume that 
$$
(X_i, d_i, \meas_i) \stackrel{mGH}{\to} (\mathbb{S}^n, d_{\mathbb{S}^n}, \meas)
$$
for some Borel probability measure $\meas$ on $\mathbb{S}^n$.
Then Proposition \ref{prop:open} shows that $(X_i, d_i, \mathcal{H}^n) $ is also a $\RCD(n-1, n)$ space for any sufficiently large $i$. Moreover, 
$$
(X_i, d_i, \mathcal{H}^n) \stackrel{mGH}{\to} (\mathbb{S}^n, d_{\mathbb{S}^n}, \mathcal{H}^n).
$$
Therefore for any sufficiently large $i$ we obtain $\mathcal{H}^n(X)\geq (1-\epsilon_n)\mathcal{H}^n(\s^n)$ and Theorem \ref{prop:toprigidity} implies that $X_i$ is homeomorphic to $\mathbb{S}^n$, which is a contradiction. 
\end{proof}

We can now prove Corollary \ref{thm:eigenhomeo};

\begin{proof}[Proof of Corollary \ref{thm:eigenhomeo}]

The proof follows the same lines as Theorem \ref{mthm}. Assume that the statement does not hold, then there exists a sequence $(X_i, d_i, \meas_i)$ of $\RCD(n-1, n)$ spaces with $\meas_i(X_i)=1$, such that $X_i$ is not homeomorphic to $\mathbb{S}^n$ and
$$
\lim_{i \to \infty}\lambda_{n+1}(X_i, d_i, \meas_i)= n.
$$
With no loss of generality we can assume that the sequence $(X_i, d_i, \meas_i)$ mGH-converges to a $\RCD(n-1, n)$ space $(X, d, \meas)$.
Then the spectral convergence result proved in \cite[Thm.7.8]{GigliMondinoSavare} shows that 
$$
\lambda_{n+1}(X, d, \meas)=\lim_{i \to \infty}\lambda_{n+1}(X_i, d_i, \meas_i)=n.
$$
Then Theorem \ref{thm:ket} yields that $(X, d)$ is isometric to $(\mathbb{S}^n, d_{\mathbb{S}^n})$. In particular since $\mathrm{rad}(X_i, d_i) \to \mathrm{rad}(\mathbb{S}^n, d_{\mathbb{S}^n})=\pi$, Theorem \ref{mthm} yields that $X_i$ is homeomorphic to $\mathbb{S}^n$ for any sufficiently large $i$, which is a contradiction. 
\end{proof}

\begin{remark}
Thanks to the intrinsic Reifenberg theorem \cite[Thm.A.1.2 and Thm.A.1.3]{CheegerColding1}, it is easy to check that all topological sphere theorems stated above can be improved to ``bi-H\"older homeomorphism''. Moreover we can choose \textit{any} $\alpha \in (0, 1)$ as a H\"older exponent. For reader's convenience let us write down the precise statement.
For all $n \in \mathbb{N}$, $r>0$ and $\epsilon >0$, let us denote by $\mathcal{M}(n, r, \epsilon)$ the set of all isometry classes of compact metric spaces $(X, d)$ with $d_{GH}(B_t(x), B_t(0_n))\le \epsilon t$ for all $x \in X$ and all $t\le r$.
Then we have:
\begin{itemize}
\item for all $\alpha \in (0, 1)$ there exist positive constants $\epsilon_0:=\epsilon_0(n ,\alpha, r)>0$ and $\delta_0(n, \alpha, r)>0$ such that if two compact metric space $Z_i \in \mathcal{M}(n, r, \epsilon_0)$ satisfies $d_{GH}(Z_1, Z_2)<\delta_0$, then there exists a homeomorphism $\Phi:Z_1 \to Z_2$ such that $\Phi$ and $\Phi^{-1}$ are $\alpha$-H\"older continuous maps.
\end{itemize}
Although we used the almost maximality of the volume (\ref{eq:max}) in the proof of Theorem \ref{prop:toprigidity} in order to simplify our argument, by an argument similar to the proof of \cite[Thm.5.11]{CheegerColding1} with corresponding almost rigidity results in \cite{DePhilippisGigli}, we see that if a non-collapsed $\RCD(K, n)$ spaces $(X, d_X, \mathcal{H}^n)$ satisfies $X=\mathcal{R}_n$, then for all $\epsilon>0$ there exist positive constants $r>0$ and $\delta>0$ such that if a non-collapsed $\RCD(K, n)$ space $(Y, d_Y, \mathcal{H}^n)$ satisfies $d_{GH}(X, Y)\le \delta$, then $(Y, d_Y) \in \mathcal{M}(n, r, \epsilon)$. See also \cite{KM}.
\end{remark}

\section{Improvement to Einstein stratified spaces}

The previous results can be improved to the case of certain (smoothly) stratified spaces. In order to do that, we briefly recall some notions about such spaces, by mostly referring to \cite{MondPhD} and \cite{BKMR} for the precise definitions. 

A (compact) stratified space $X$ is a (compact) topological space which admits a decomposition in strata
$$X = \bigsqcup_{j=0}^n \Sigma^{j}(X)$$
such that for each $j=0,\ldots n$, $\Sigma^{j}(X)$ is a smooth manifold of dimension $j$, $\Sigma^n(X)$ is open and dense in $X$ and $\Sigma^{n-1}(X)=\emptyset$. We denote the higher dimension stratum $\Sigma^n(X)$ as $X^{\mbox{\tiny{reg}}}$, the regular set of $X$, and refer to $n$ as the dimension of $X$. We define the singular set of $X$: 
$$ X^{\mbox{\tiny{sing}}}= \bigsqcup_{j=0}^{n-2}\Sigma^{j}(X).$$
For $j < (n-1)$, $\Sigma^{j}(X)$ is called the singular stratum of dimension $j$. For each point in $\Sigma^{j}(X)$  there exists a neighbourhood $\mathcal{U}_x$ homeomorphic to the product of an Euclidean ball in $\R^j$ and a truncated cone over a compact stratified space $B_r(0_j)\times C_{[0,r)}(Z_j)$. We refer to $Z_j$ as the \emph{link} of the stratum $\Sigma^j(X)$. 

By induction on the dimension, a stratified space $X$ can be endowed with an iterated edge metric $g$, which is a Riemannian metric on $X^{\mbox{\tiny{reg}}}$ with the appropriate asymptotics close to each singular stratum: by denoting  $k_j$ an iterated edge metric on the link $Z_j$, there exist positive constants $\Lambda$ and $\gamma$ such that in a neighbourhood $\mathcal{U}_x$ of a point $x \in \Sigma^{j}(X)$ we have
\begin{equation}
\label{eq:ItEdge}
|\varphi_x^*g-(h+dr^2+r^2k_j)| \leq \Lambda r^{\gamma},
\end{equation}
where $h$ is the standard Riemannian metric on $\R^j$ and $\varphi_x$ is the homeomorphism between the product $B_r(0_j)\times C_{[0,r)}(Z_j)$ and the neighbourhood  $\mathcal{U}_x$. 

In the case of the codimension $2$ stratum, the link is a compact stratified space of dimension $1$, thus a circle. As a consequence, for each point $x \in \Sigma^{n-2}(X)$ there exists $\alpha_x \in (0, +\infty)$ such that the metric $g$ is asymptotic, in the sense of \eqref{eq:ItEdge}, to:
$$ h+ dr^2+\left(\frac{\alpha_x}{2\pi} \right)r^2d\theta^2,$$
on $\R^{n-2}\times C(\s^1)$. We refer to $\alpha_x$ as the angle of $\Sigma^{n-2}(X)$ at $x$. 

The iterated edge metric $g$ gives rise to a length structure and a distance $d_g$ on $X$, and to a Riemannian measure $\mu_g$ which in the compact case is Ahlfors-regular and finite. Note that the measure of the singular strata is zero and, as in the case of smooth manifolds, $\mu_g$ coincides with the Hausdorff measure.

Moreover, thanks to the definition of the distance and iterated edge metric, we know that each point $x \in \Sigma^{j}(X)$ admits a unique tangent cone $C(S_x)$ over the tangent sphere at $x$. It is defined by the pGH limit as $\eps$ goes to zero: 
$$(X, \eps^{-1}d_g, x) \stackrel{pGH}\longrightarrow (C(S_x), d_C, o),$$
where $o$ is the vertex of the cone. The tangent sphere is a compact stratified space of dimension $(n-1)$ given by the $(j-1)$-spherical suspension of the link
$$S_x=\left[0, \frac{\pi}{2} \right] \times \s^{j-1} \times Z_j.$$
Since the convergence is smooth on the regular sets, the tangent sphere is endowed with a double warped product metric:
\begin{equation}\label{eq:double}
h_x= d\psi^2+\cos^2(\psi) g_{\s^{j-1}}+\sin^2(\psi) k_j.
\end{equation}
Moreover, the smoothness of the convergence and the fact that singular sets have null measures imply that pGH-convergence can be replaced by \textit{pmGH-convergence}. 

Note that for $x  \in \Sigma^{0}(X)$, the tangent sphere coincides with the link of $\Sigma^{0}(X)$. If $x$ belongs to $\Sigma^{1}(X)$ then $S_x$ is a spherical suspension of the form $[0,\pi] \times Z_1$ endowed with the warped product  metric $h_x=d\psi^2+\sin^2(\psi)k_1$. 

Also observe that a point belongs to the regular set if and only if its tangent sphere is isometric to the round sphere $\s^{n-1}$. 

In view of the following, we need information about how the regularity of $X$ affects the regularity of tangent spheres. This is stated in the following.

\begin{lemma}
\label{lem:Codim2sing}
Let $(X, g)$ be an $n$-dimensional compact stratified space endowed with an iterated edge metric $g$. If $\Sigma^{n-2}(X)=\emptyset$, then for any $x \in X$ the tangent sphere $S_x$ does not carry a singular stratum of codimension $2$. 
\end{lemma}
\begin{proof}
Assume $x \in  \Sigma^j(X)$ for $j \in \{0,\ldots n-3\}$ and consider $Z_j$ the link of $\Sigma^j(X)$. Denote by $d_j=n-j-1$ its dimension and by $\varphi_x$ the homeomorphism between a neighbourhood of $x$ and the product $B_r(0_j)\times C_{[0,r)}(Z_j)$. We first observe that $Z_j$ does not carry any singular stratum of codimension $2$. Assume by contradiction that there exists $z \in \Sigma^{d_j-2}(Z_j)$: then $z$ has a neighbourhood homeomorphic to $B_s(0_{d_j-2})\times C_{[0,s)}(\s^1)$. Denote by $\bar{p}$ the point of coordinates $(v,s,z)$ in $B_r(0_{j})\times C_{[0,r)}(Z_j)$ and $\bar x =\varphi_x(\bar p) \in X$. Then $\bar x$ has a neighbourhood homeomorphic to $B_{\rho}(0_{n-2})\times C(\s^1)$. As a consequence, $\bar x$ belongs to $\Sigma^{n-2}(X)$, which contradicts the assumption $\Sigma^{n-2}(X)=\emptyset$.

We next show that $Z_j$ not having any singular stratum of codimension $2$, the same holds for the tangent sphere $S_x$. If $j=0$, $S_x$ coincides with $Z_j$ and thus have the same singular set. If $j=1$, $S_x$ is the warped product $([0,\pi] \times Z_1, d\psi^2+\sin^2(\psi)k_1)$. Its singular set $S_x^{\mbox{\tiny{sing}}}$ is composed of the subsets $(0,\pi)\times Z_1^{\mbox{\tiny{sing}}}$ and $\{0,\pi\}\times Z_1$. As for $(0,\pi)\times Z_1^{\mbox{\tiny{sing}}}$, it only generates singularities of the same codimension as the ones of $Z_1^{\mbox{\tiny{sing}}}$, that is at least 3. The singularities at $\{0\} \times Z_1$ and $\{\pi\}\times Z_1$ carry a neighbourhood homeomorphic to $C(Z_1)$ and as a consequence have codimension 0 in $S_x$. Therefore $\Sigma^{n-2}(S_x)=\emptyset$. Similarly, for $j\in \{2, \ldots , n-3\}$, the  singular set of $S_x$ is given by : 
\begin{itemize}
\item[•] the product $(0, \pi/2)\times Z_j^{\mbox{\tiny{sing}}}$ which has codimension at least 3 because we already know $\Sigma^{d_j-2}(Z_j)=\emptyset$; 
\item[•] $\{0\} \times \s^{j-1} \times Z^j$: any point in this product has a neighbourhood homeomorphic to $B(0_j)\times Z_j$, then for the same reason singularities have codimension at least 3; 
\item[•] $\{\pi/2\} \times \s^{j-1} \times Z^j$: any point belonging to this set has a neighbourhood homeomorphic to $\s^{j-1}\times C(Z_j)$. This gives singularities of codimension $n-j$ in $S_x$ and since $j \leq n-3$, the codimension is at least 3. 
\end{itemize}
We have shown that for any $x \in X$ we have $\Sigma^{n-2}(S_x)=\emptyset$, as we wished. \end{proof}

Thanks to \cite{BKMR} we know the following: 

\begin{theorem*}
A compact stratified space $(X, d_g, \mu_g)$ of dimension $n$ endowed with an iterated edge metric $g$ is a $\RCD(K,N)$ space for $K\in \R$, $N \geq 1$, if and only if $n \leq N$, $\mathrm{Ric}_g \geq K$ on $X^{\mbox{\tiny{reg}}}$  and the angles along the singular stratum of codimension $2$ are smaller than or equal to $2\pi$. 
\end{theorem*}

Since $\mu_g$ is the Hausdorff measure, a $\RCD(K,n)$ compact stratified space of dimension $n$ is also a \emph{non-collapsed} $\RCD$ space. 

An easy consequence of the definition of the iterated edge metric is the following:

\begin{lemma}
\label{lem:RicTg}
Let $(X, d_g, \mu_g)$ be a $\RCD(K,n)$ compact stratified space of dimension $n$ endowed with an iterated edge metric $g$. Then for every $x \in X$, the tangent sphere $(S_x, d_{h_x}, \mu_{h_x})$ at $x$ is a non-collapsed $\RCD(n-2,n-1)$ space. Moreover, if for $K \geq 0$ $||\mathrm{Ric}_g||\leq K$ holds on the regular set, we also have $\mathrm{Ric}_{h_x}=(n-2)$ on $S_x^{\mbox{\tiny{reg}}}$. 
\end{lemma}

\begin{proof}
Because of the definition of the tangent cone, we know that if $\mathrm{Ric}_g \geq K$ on $X^{\mbox{\tiny{reg}}}$, then $(C(S_x), ds^2+s^2h_x)$ has non-negative Ricci tensor on its regular part. As a consequence, for each tangent sphere $\mathrm{Ric}_{h_x}\geq (n-2)$ on $S_x^{\mbox{\tiny{reg}}}$ (see Lemma 2.1 in \cite{Mondello}). Analogously, if $||\mathrm{Ric}_g|| \leq K$ on $X^{\mbox{\tiny{reg}}}$ the tangent cone is Ricci flat and $\mathrm{Ric}_{h_x}=(n-2)$ on $S_x^{\mbox{\tiny{reg}}}$. 

For every $x$, $S_x$ cannot carry a stratum of codimension $2$ with angles larger than $2\pi$. The argument is the same as in Lemma \ref{lem:Codim2sing}. If for every $x \in \Sigma^{n-2}(X)$ the angle $\alpha_x$ is smaller than or equal to $2\pi$, then the same holds for every link $Z_j$ of dimension $d_j$: the angles along $\Sigma^{d_j-2}(Z_j)$ belongs to $(0,2\pi]$. Now, if $x \in \Sigma^j(X)$, singularities of codimension $2$ of $S_x$ are determined by the singularities of codimension $2$ in $Z_j$: as a consequence, if $p \in \Sigma^{n-3}(S_x)$, then $\alpha_p \in (0, 2\pi]$.

Then for any $x \in X$ the tangent sphere at $x$ satisfies the assumptions of the previous theorem and  $(S_x,d_{h_x}, \mu_{h_x})$ is a $\RCD(n-1,n-2)$ space. 
\end{proof}

We are now in position to prove the following: 

\begin{theorem}[Sphere theorem for Einstein stratified spaces]\label{thm:strati}
For all $n \in \mathbb{N}_{\ge 2}$ there exists a positive constant $\epsilon_n>0$ such that the following holds. Let $(X, g)$ be  compact $n$-dimensional stratified space endowed with an iterated edge metric $g$ satisfying  $\mathrm{Ric}_g \equiv n-1$ on the regular set. Assume that $\Sigma^{n-2}(X) = \emptyset$. If $\mu_g(X) \ge (1-\epsilon_n)\mathcal{H}^n(\mathbb{S}^n)$, then $(X, d_g)$ is isometric to $(\mathbb{S}^n, d_{\mathbb{S}^n})$.
\end{theorem}
\begin{proof}
Let us first check the theorem in the case of $(X, g)$ not having a singular set, that is $(X, g)$ is a smooth manifold. The proof is done by a contradiction. If the assertion does not hold, then there exists a sequence of $n$-dimensional closed Riemannian manifolds $(M^n_i, g_i)$ such that $\mathrm{Ric}_{M^n_i}^{g_i} \equiv n-1$, that $\mu_{g_i}(M^n_i)\to \mathcal{H}^n(\mathbb{S}^n)$ and that $(M^n_i, g_i)$ is not isometric to $(\mathbb{S}^n, g_{\mathbb{S}^n})$. Then applying the smooth convergence result \cite[Thm.7.3]{CheegerColding1} yields that $(M^n_i, g_i)$ converge smoothly to $(\mathbb{S}^n, g_{\mathbb{S}^n})$. In particular $(M^n_i, g_i)$ is simply connected and has positive curvature operator for suficiently large $i$. Then by a theorem of \cite{Tachibana} $(M^n_i, g_i)$ has constant sectional curvature. Thus $(M^n_i, g_i)$ is isometric to $(\mathbb{S}^n, g_{\mathbb{S}^n})$, which is a contradiction.

For all $n \in \mathbb{N}_{\geq 2}$ take a positive constant $\epsilon_n>0$ such that the smooth version of Theorem \ref{thm:strati} holds. 
Next let us fix a compact $n$-dimensional stratified space $(X, g)$ satisfying $\mathrm{Ric}_g \equiv n-1$ on the regular set and $\Sigma^{n-2}(X) = \emptyset$. Our goal is to prove that if $\mu_g (X) \geq (1-\hat{\epsilon}_n)\mathcal{H}^n(\mathbb{S}^n)$, then $(X, d_g)$ is isometric to $(\mathbb{S}^n, d_{\mathbb{S}^n})$, where $\hat{\epsilon}_n:=\min \{\epsilon_{i}, 1 \le i \le n\}$ and $\epsilon_1:=0$.

The proof is done by induction.
For $n=2$, our assumption $\Sigma^{n-2}(X)=\emptyset$ yields that $(X, g)$ has no singular set, thus $(X, g)$ is a smooth Riemannian manifold. By definition of $\epsilon_2$ we get the desired statement.

As for $n \geq 3$, let $x \in X$ and consider the tangent cone $(C(S_x), ds^2+s^2h_x, o)$ at $x$. By an argument similar to (\ref{eq:maxball}), we see that for all points $x \in X$ and $r\in (0,\pi]$ we have $\mu_g(B_r(x))=\mathcal{H}^n(B_r(x))\geq (1-\hat{\epsilon}_n)\omega_n r^n$. By rescaling the metric by a factor $r^{-2}$ and by letting $r$ go to zero, $(X,r^{-2}g, \mathcal{H}^n, x)$ pmGH-converges to the tangent cone $(C(S_x), ds^2+s^2h_x, \mathcal{H}^n, o)$. As a consequence, we have $\mathcal{H}^n(B_1(o)) \geq (1-\hat{\epsilon}_n)\omega_n$. Since $\mu_{h_x}(S_x)=n\mathcal{H}^n(B_1(o))$, we obtain
$$\mu_{h_x}(S_x) \geq (1-\hat{\epsilon}_{n})\mathcal{H}^{n-1}(\s^{n-1})\geq (1-\hat{\epsilon}_{n-1})\mathcal{H}^{n-1}(\s^{n-1}).$$ 
On the other hand, by Lemma \ref{lem:RicTg}, we also know that $(S_x,h_x)$ is a compact stratified space of dimension $(n-1)$ with $\mathrm{Ric}_{h_x}\equiv (n-2)$ on its regular set. Moreover, since $\Sigma^{n-2}(X)=\emptyset$, Lemma \ref{lem:Codim2sing} ensures that $S_x$ does not have any singular stratum of codimension $2$. Then by the assumption on the induction, $(S_x, d_{h_x})$ is isometric to the round sphere $(\mathbb{S}^{n-1}, d_{\mathbb{S}^{n-1}})$. As a consequence we have proven that for any $x \in X$, $x$ is a regular point. This proves that $X^{\mbox{\tiny{sing}}}=\emptyset$, thus $(X, g)$ is a smooth Riemannian manifold. By definition of $\epsilon_n$, $(X, d_g)$ is isometric to $(\mathbb{S}^n, d_{\mathbb{S}^n})$. 
\end{proof}

\begin{remark} In the theorem above the assumption $\Sigma^{n-2}(X)=\emptyset$ is essential. Consider for all $a \in (0,1)$, the $(n-2)$-spherical suspension of $\s^1(a):= (\s^1,a^2d\theta^2)$ defined by: 
$$ X =\left[0, \frac{\pi}{2} \right] \times \s^{n-2} \times \s^1(a),$$
endowed with the double warped product metric $g$ as in (\ref{eq:double}). $(X, g)$ is a compact $n$-dimensional stratified space with the iterated edge metric $g$ and $\mathrm{Ric}_g \equiv (n-1)$ on $X^{\mbox{\tiny{reg}}}$ and a non-empty, codimension $2$ singular stratum given by $\{\pi/2\} \times \s^{n-2} \times \s^1(a)$. If the angle $\alpha = 2\pi a$ along $\Sigma^{n-2}(X)$ is close to $2\pi$, then the volume of $X$ is close to the one of $\s^n$, but $(X, g)$ cannot be isometric to $(\s^n, d_{\s^n})$ because of its singular stratum of codimension $2$. 
\end{remark}

We now consider stratified spaces that are not necessarily compact. Note that even in this case tangent spheres are compact stratified space defined as above, and Lemmas \ref{lem:Codim2sing} and \ref{lem:RicTg} still hold. In presence of a codimension $2$ stratum and a two-side Ricci bound, the previous theorem allows us to obtain the following:

\begin{cor}\label{cor:reg}
For all $n \in \mathbb{N}_{\ge 2}$ there exists a positive constant $\epsilon_n>0$ such that the following holds. Let $(X, g)$ be an $n$-dimensional stratified space endowed with an iterated edge metric $g$ such that $\mathrm{Ric}_g$ is two-side bounded on the regular set of $X$. If a point $x \in X$ satisfies
$$
\lim_{r \to 0^+}\frac{\mu_g (B_r(x))}{\omega_nr^n}\ge 1-\epsilon_n,
$$
then either $x$ belongs to $X^{\mbox{\tiny{reg}}}$, or $x \in \Sigma^{n-2}(X)$ and $\alpha_x \geq 2\pi(1 -\epsilon_n)$. 
\end{cor}
\begin{proof}
Fix $\epsilon_n$ as in Theorem $\ref{thm:strati}$ and assume that 
\begin{equation}
\label{eq:limVol}
\lim_{r \to 0^+}\frac{\mu_g (B_r(x))}{\omega_n r^n}\geq 1-\epsilon_n.
\end{equation}
If $x$ belongs to $\Sigma^{n-2}(X)$, thanks to the definition of the tangent sphere and its metric we know that 
$$\lim_{r\rightarrow 0^+}\frac{\mu_g(B_r(x))}{\omega_n r^n}=\frac{\alpha_x}{2\pi}.$$
Therefore if (\ref{eq:limVol}) holds, we obtain $\alpha_x \geq 2\pi(1-\epsilon_n)$. 

Consider a point $x \in X \setminus \Sigma^{n-2}(X)$. Observe that the tangent sphere $(S_x,h_x)$ at $x$ is a compact stratified space of dimension $(n-1)$, and with the same argument as in Lemma $\ref{lem:RicTg}$, the metric $h_x$ satisfies $\mathrm{Ric}_{h_x}\equiv (n-2)$. Moreover, since $x$ does not belong to $\Sigma^{n-2}(X)$, Lemma \ref{lem:Codim2sing} can be easily adapted to show that $S_x$ does not carry any singular stratum of codimension $2$.  By the same argument as in the previous theorem, letting $r$ go to zero in $(B_r(x),r^{-2}g,x)$ leads to the volume estimate $\mathcal{H}^n(B_1(o))\geq (1- \epsilon_{n})\omega_{n}$ and as above we obtain: 
$$\mu_{h_x}(S_x)\geq (1-\epsilon_{n})\mathcal{H}^{n-1}(\s^{n-1}).$$ 
As a consequence, Theorem \ref{thm:strati} applied to $S_x$ shows that the tangent sphere at $x$ is isometric to $\s^{n-1}$ and $x$ belongs to the regular set. 
\end{proof}

\begin{remark}
In Corollary \ref{cor:reg}
the assumption on two side bounds on the Ricci curvature on the regular set is essential because for any $n \in \mathbb{N}_{\ge 3}$, taking $r \in (0, 1)$ which is sufficiently close to $1$, let us consider a compact $n$-dimensional stratified space $X:=[0,\pi] \times \mathbb{S}^{n-1}$ with the warped product metric $g=d\varphi^2+\sin^2(\varphi)r^2g_{\s^{n-1}}$. Then it satisfies 
\begin{itemize}
\item for all $x \in X$, $\lim_{t \to 0^+}\frac{\mu_g (B_t(x))}{\omega_nt^n}$ is close to $1$;
\item $\mathrm{Ric}_g \ge n-1$ on the regular set;
\item there is no upper bound on $\mathrm{Ric}_g$. 
\end{itemize}
In this case, all singular points of $X$ belongs to $\Sigma^{0}(X)$.
\end{remark}


\end{document}